\documentclass[12pt]{amsart}

\usepackage[all]{xy}

\theoremstyle{plain}
\newtheorem{thm}{Theorem}[section]
\newtheorem{prop}[thm]{Proposition}
\newtheorem{lemma}[thm]{Lemma}

\newtheorem{cor}[thm]{Corollary}
\newtheorem{question}[thm]{Question}

\theoremstyle{definition}

\theoremstyle{remark}
\newtheorem{rem}[thm]{Remark}

\newcommand{\HH}{\mathrm{H}}

\begin{document}

\title{Universal deformation rings and dihedral $2$-groups}

\author{Frauke M. Bleher}
\address{Department of Mathematics\\University of Iowa\\
Iowa City, IA 52242-1419}
\email{fbleher@math.uiowa.edu}
\thanks{The author was supported in part by  
NSA Grant H98230-06-1-0021 and NSF Grant DMS06-51332.}

\begin{abstract}
Let $k$ be an algebraically closed field of characteristic $2$, and let $W$ be the ring of infinite Witt vectors over $k$. 
Suppose $D$ is a dihedral $2$-group.  We prove that the universal deformation ring 
$R(D,V)$ of an endo-trivial $kD$-module $V$ 
is always isomorphic to $W[\mathbb{Z}/2\times\mathbb{Z}/2]$.
As a consequence we obtain a similar result for modules $V$ with stable endomorphism ring $k$
belonging to an arbitrary nilpotent block 
with defect group $D$.
This confirms for such $V$ conjectures on the ring structure
of the universal deformation ring of $V$ which had previously been shown for 
$V$ belonging to cyclic blocks or to blocks with Klein four defect groups.
\end{abstract}

\maketitle


\section{Introduction}
\label{s:intro}
\setcounter{equation}{0}

Let $k$ be an algebraically closed field of positive characteristic $p$, let $W=W(k)$ be the ring
of infinite Witt vectors over $k$, and let $G$ be a finite group.
There are various classical results in the literature concerning the lifting of finitely generated
$kG$-modules over complete local commutative Noetherian rings with residue field $k$,
such as Green's liftability theorem. 
To understand and generalize
these results, it is useful to reformulate them in terms of  deformation rings. For example,
Alperin has proved in \cite{alpendotrivial}  that in case $G$ is a $p$-group, every endo-trivial
$kG$-module $V$ can be lifted to an endo-trivial $WG$-module. This can be reformulated as saying 
that for every such $V$, there is a surjection from the universal deformation 
ring $R(G,V)$ to $W$. A natural question is then to determine $R(G,V)$ itself.
In this paper, we determine the universal deformation rings $R(D,V)$ for all endo-trivial
$kD$-modules $V$ when $k$ has characteristic $p=2$ and $D$ is a dihedral $2$-group.

For arbitrary $p$ and $G$, a finitely generated
$kG$-module $V$ is called endo-trivial if the $kG$-module
$\mathrm{Hom}_k(V,V)\cong V^*\otimes_k V$ 
is isomorphic to a direct
sum of the trivial simple $kG$-module $k$ and a projective $kG$-module. 
Endo-trivial modules play an important role
in the modular representation theory of finite groups, in particular in the context of
derived equivalences and stable equivalences of block algebras, and also as building blocks for
the more general endo-permutation modules, which for many groups are the sources of the
simple modules (see e.g. \cite{dade,thevenaz}).
In \cite{carl1.5,carl2}, Carlson and Th\'{e}venaz classified all endo-trivial $kG$-modules
when $G$ is a $p$-group. Since by \cite{carl1}, the endo-trivial $kG$-modules $V$ of a $p$-group $G$ 
are precisely the modules whose stable endomorphism ring is one-dimensional over $k$, it follows
by \cite[Prop. 2.1]{bc} that $V$ has a well-defined universal deformation ring $R(G,V)$. 

The topological ring $R(G,V)$ is universal with respect to deformations of $V$ over complete local 
commutative Noetherian rings $R$ with residue field $k$. A deformation of $V$ over such a ring
$R$ is given by the isomorphism class of a finitely generated $RG$-module $M$ which is free
over $R$, together 
with a $kG$-module isomorphism $k\otimes_R M\to V$ (see \S\ref{s:prelim}). 
Note that all these rings $R$, including $R(G,V)$, have a natural structure as $W$-algebras.

In number theory, the main motivation for studying universal
deformation rings for finite groups is to provide evidence for and counter-examples to various 
possible conjectures concerning ring theoretic properties of universal deformation rings for 
profinite Galois groups. The idea is that universal deformation rings for finite groups can be more easily
described using deep results from modular representation theory
due to Brauer,  
Erdmann \cite{erd}, Linckelmann \cite{linckel,linckel2}, 
Butler-Ringel \cite{buri} and others. 
Moreover, the results in \cite{lendesmit} show that if $\Gamma$ is a profinite group and
$V$ is a finite dimensional $k$-vector space with a continuous $\Gamma$-action
which has a universal deformation ring, then $R(\Gamma,V)$ is
the inverse limit of the universal deformation rings $R(G,V)$ when 
$G$ runs over all finite discrete quotients of $\Gamma$ through which the $\Gamma$-action on $V$ factors. Thus to answer questions about 
the ring structure of $R(\Gamma,V)$, it is natural to first consider the case when $\Gamma=G$ is finite. 
Later in the introduction 
we will discuss some number 
theoretic problems which originate from considering how our results for finite groups arise from 
arithmetic.

Suppose now that $G$ is an arbitrary finite group and $V$ is a $kG$-module such that the stable endomorphism ring
$\underline{\mathrm{End}}_{kG}(V)$ is one-dimensional over $k$, i.e. $V$ has a well-defined
universal deformation ring $R(G,V)$. The results in \cite{bc} led to 
the following question relating the universal deformation rings $R(G,V)$ to the local
structure of $G$ given by defect groups of blocks of $kG$.

\begin{question}
\label{qu:main} 
Let $B$ be a block of $kG$ with defect group $D$, and suppose $V$ is a finitely generated $kG$-module with stable endomorphism ring $k$ 
such that the unique $($up to isomorphism$)$ non-projective indecomposable summand of $V$ belongs to $B$. Is the universal deformation 
ring $R(G,V)$ of $V$ isomorphic to a subquotient ring of the group ring $WD$?
\end{question}

The results in \cite{bc,bl,3sim} show that  this question has a positive answer
in case $B$ is a block with cyclic defect groups, i.e. a block of finite representation type,
or a tame block with Klein four defect groups, or a tame block with dihedral defect groups
which is Morita equivalent to the principal $2$-modular block of a finite simple group.
For the latter type of blocks, there are precisely three isomorphism classes of simple modules.

In \cite{bc4.9,bc5}, it was shown that if $p=2$, $G$ is the symmetric group $S_4$ and 
$E$ is a $2$-dimensional simple $kS_4$-module then
$R(G,E)\cong W[t]/(t^2,2t)$, giving an example of a universal deformation ring which is not a complete intersection, thus answering a question of M. Flach \cite{flach}. 
A new proof of this result has been given in \cite{cantrememberhisname} using only elementary obstruction calculus.
In \cite{bc4.9,bc5}, it was additionally shown that this example arises
from arithmetic in the following way.
There are infinitely many real quadratic fields $L$ such that the Galois group $G_{L,\emptyset}$ of the 
maximal totally unramified extension of $L$ surjects onto $S_4$ and 
$R(G_{L,\emptyset},E)\cong R(S_4,E)\cong  W[t]/(t^2,2t)$ 
is not a complete intersection, where $E$ is viewed as a module of $G_{L,\emptyset}$ via inflation. 
The universal deformation rings in \cite{bc,bl} 
are all complete intersections, 
whereas the results in \cite{3sim} provide an infinite series of $G$ and $V$ for which $R(G,V)$ is not a complete intersection.

In this paper, we consider the endo-trivial modules for the group ring $kD$ when $k$ has characteristic $2$ 
and $D$ is a dihedral $2$-group of order at least $8$.
Note that $kD$ is its own block and the trivial simple module $k$ is the unique  simple $kD$-module
up to isomorphism. 
Our main result is as follows, where $\Omega$ 
denotes the syzygy, or Heller, operator (see for example \cite[\S 20]{alp}).

\begin{thm}
\label{thm:supermain}
Let $k$ be an algebraically closed field of characteristic $2$,
let $d\ge 3$, and let $D$ be a dihedral group of order $2^d$. 
Suppose $V$ is a finitely generated  endo-trivial
$kD$-module. 
\begin{itemize}
\item[i.] If $V$ is indecomposable
and $\mathfrak{C}$ is the component of the stable Auslander-Reiten quiver of $kD$ 
containing $V$,
then $\mathfrak{C}$ contains either $k$ or $\Omega(k)$,
and all modules belonging to $\mathfrak{C}$ are endo-trivial. 

\item[ii.]
The universal deformation ring $R(D,V)$ is isomorphic to $W[\mathbb{Z}/2\times \mathbb{Z}/2]$.
Moreover, every universal lift $U$ of $V$ over $R=R(D,V)$ is endo-trivial in the sense that
the $RD$-module $U^*\otimes_RU$ is isomorphic to the direct sum of the trivial $RD$-module
$R$ and a free $RD$-module.
\end{itemize}
\end{thm}

In particular, $R(D,V)$ is always a complete intersection 
and isomorphic to a quotient ring of the group ring $WD$.

It is a natural question to ask whether Theorem \ref{thm:supermain} can be used to construct
deformation rings arising from arithmetic. More precisely, 
let $L$ be a number field, let $S$ be a finite set of places of $L$, and let $L_S$ be the maximal 
algebraic extension of $L$ unramified outside $S$. Denote by $G_{L,S}$ the Galois group of 
$L_S$ over $L$.
Suppose $k$ has characteristic $p$, $G$ is a finite group and $V$ is a finitely generated 
$kG$-module with stable endomorphism ring $k$.
As in \cite{bc5}, one can ask whether there are $L$ and $S$ such 
that there is a surjection $\psi:G_{L,S}\to G$ which induces an isomorphism 
of deformation rings $R(G_{L,S},V)\to R(G, V)$ 
when $V$ is viewed as a representation for $G_{L,S}$ via $\psi$. It was shown in 
\cite{bc5} that a sufficient condition for $R(G_{L,S},V)\to R(G, V)$ to be an isomorphism for all 
such $V$ is that  $\mathrm{Ker}(\psi)$ has no non-trivial pro-$p$ quotient. If this condition is
satisfied, we say the group $G$ caps $L$ for $p$ at $S$.

As mentioned earlier, this arithmetic problem was considered in \cite{bc5} for the symmetric
group $S_4$ in case $p=2$. In fact, it was shown that there are infinitely many real quadratic 
fields $L$ such that  $S_4$ caps $L$ for $p=2$ at $S=\emptyset$. Since the Sylow $2$-subgroups 
of $S_4$ are isomorphic to a dihedral group $D_8$ of order $8$, this can be used to show that 
$D_8$ caps infinitely many sextic fields $L'$ for $p=2$ at $S=\emptyset$.
In particular, $R(G_{L',\emptyset},V)\cong R(D_8, V)\cong 
W[\mathbb{Z}/2\times\mathbb{Z}/2]$ for all endo-trivial $kD_8$-modules $V$. Since the fields $L'$ have
degree $6$ over $\mathbb{Q}$, this raises the
question of whether one can replace $L'$ by smaller degree extensions. 
Another question is if one can find similar results for dihedral groups $D$ of arbitrary $2$-power order.
As in the proof of \cite[Thm. 3.7(i)]{bc5}, one can show that $D$ does not cap $\mathbb{Q}$ for $p=2$ 
at  any finite set $S$ of rational primes.
Hence the best possible results one can expect to be valid for all endo-trivial $kD$-modules 
should involve extensions $L$ of $\mathbb{Q}$ of degree at least $2$.

We now discuss the proof of Theorem \ref{thm:supermain}.
As stated earlier, the endo-trivial $kD$-modules are
precisely the $kD$-modules whose stable endomorphism ring is one-dimensional over
$k$. The results of \cite[\S5]{CT} show that the group $T(D)$ of equivalence classes 
of endo-trivial $kD$-modules is generated by the classes of the relative syzygies of the
trivial simple $kD$-module $k$. 
To prove part (i) of Theorem \ref{thm:supermain} we relate this
description for indecomposable endo-trivial $kD$-modules to their location in the
stable Auslander-Reiten quiver of $kD$.
For part (ii) of Theorem \ref{thm:supermain},
suppose $D=\langle \sigma,\tau\rangle$ 
where $\sigma$ and $\tau$ are two elements of order $2$ and $\sigma\tau$ has order $2^{d-1}$, 
and let $V$ be an indecomposable endo-trivial $kD$-module.
We prove that there exists a continuous local 
$W$-algebra homomorphism 
$\alpha:W[\mathbb{Z}/2\times\mathbb{Z}/2]\to R(D,V)$ by considering restrictions of $V$ to
$\langle\sigma\rangle$ and $\langle\tau\rangle$. 
We then analyze the $kD$-module structures
of all lifts of $V$ over the dual numbers  $k[\epsilon]/(\epsilon^2)$ to show that $\alpha$ is
in fact surjective. 
Using the ordinary irreducible representations of $D$, we  prove that
there are four distinct continuous $W$-algebra homomorphisms $R(D,V)\to W$  and show that
this implies that $\alpha$ is an isomorphism.

In \cite{brouepuig,puig}, Brou\'{e} and Puig introduced and studied so-called nilpotent blocks. 
Using  \cite{puig},
we obtain the following result as an easy consequence of Theorem 
\ref{thm:supermain}.

\begin{cor}
\label{cor:nilpotent}
Let $k$ and $D$ be as in Theorem \ref{thm:supermain}. Let $G$ be a finite group, and
let $B$ be a nilpotent block of $kG$ with defect group $D$. 
Suppose $V$ is a finitely generated $B$-module with stable endomorphism
ring $k$.
Then the universal deformation ring $R(G,V)$ is isomorphic to $W[\mathbb{Z}/2\times \mathbb{Z}/2]$.
\end{cor}

The paper is organized as follows:
In \S\ref{s:prelim}, we provide some background on universal deformation rings for finite
groups. In \S\ref{s:dihedral}, we study some subquotient modules of the free $kD$-module
of rank $1$ and describe lifts of two such
$kD$-modules over $W$ using the ordinary irreducible representations of $D$.
In \S\ref{s:endotrivial}, we describe the locations of the indecomposable endo-trivial $kD$-modules
in the stable Auslander-Reiten quiver of $kD$ using \cite{AC,CT}.
In \S\ref{s:udr}, we complete the proof of Theorem \ref{thm:supermain} and Corollary \ref{cor:nilpotent}.

The author would like to thank K. Erdmann and M. Linckelmann for helpful discussions on nilpotent blocks. She would also like to thank the referee for 
very useful comments which simplified some of the proofs.


\section{Preliminaries}
\label{s:prelim}
\setcounter{equation}{0}

Let $k$ be an algebraically closed field of characteristic $p>0$, let $W$ be the ring of infinite Witt 
vectors over $k$ and let $F$ be the fraction field of $W$. Let ${\mathcal{C}}$ be the category of 
all complete local commutative Noetherian rings with residue field $k$. The morphisms in 
${\mathcal{C}}$ are continuous $W$-algebra homomorphisms which induce the identity map on $k$. 

Suppose $G$ is a finite group and $V$ is a finitely generated $kG$-module. 
If $R$ is an object in $\mathcal{C}$, a finitely generated $RG$-module $M$ is called
a lift of $V$ over $R$ if $M$ is free over $R$ and $k\otimes_R M\cong V$ as $kG$-modules. Two lifts 
$M$ and $M'$ of $V$ over $R$ are said to be isomorphic if there is an $RG$-module isomorphism 
$\alpha:M\to M'$ which respects the $kG$-module isomorphisms $k\otimes_R M\cong V$ and
$k\otimes_R M'\cong V$. The isomorphism class of a lift of $V$ 
over $R$ is called a deformation of $V$ over $R$, and the set of such deformations is denoted by 
$\mathrm{Def}_G(V,R)$. The deformation functor ${F}_V:{\mathcal{C}} \to \mathrm{Sets}$
is defined to be the covariant functor which sends an object $R$ in ${\mathcal{C}}$ to 
$\mathrm{Def}_G(V,R)$.

In case there exists an object $R(G,V)$ in ${\mathcal{C}}$ and a lift $U(G,V)$ of $V$ over $R(G,V)$ 
such that for each $R$ in ${\mathcal{C}}$ and for each lift $M$ of $V$ over $R$ there is a unique 
morphism $\alpha:R(G,V)\to R$ in ${\mathcal{C}}$ such that $M\cong R\otimes_{R(G,V),\alpha}
U(G,V)$, then $R(G,V)$ is called the universal deformation ring of $V$ and the isomorphism class
of the lift $U(G,V)$ is called the universal deformation of $V$. In other words, $R(G,V)$ represents
the functor ${F}_V$ in the sense that ${F}_V$ is naturally isomorphic to 
$\mathrm{Hom}_{{\mathcal{C}}}(R(G,V),-)$. If $R(G,V)$ and the universal deformation 
corresponding to $U(G,V)$ exist, then they are 
unique up to unique isomorphism.
For more information on deformation rings see \cite{lendesmit} and \cite{maz1}.

The following four results were proved in \cite{bc} and in \cite{3sim}, respectively. Here $\Omega$ 
denotes the syzygy, or Heller, operator for $kG$ (see for example \cite[\S 20]{alp}).

\begin{prop}
\label{prop:stablend}
{\rm \cite[Prop. 2.1]{bc}}
Suppose $V$ is a finitely generated $kG$-mod\-ule with stable endomorphism ring 
$\underline{\mathrm{End}}_{kG}(V)=k$.  Then $V$ has  a universal deformation ring $R(G,V)$.
\end{prop}

\begin{lemma} 
\label{lem:defhelp}
{\rm \cite[Cors. 2.5 and 2.8]{bc}}
Let $V$ be a finitely generated $kG$-module with stable endomorphism ring 
$\underline{\mathrm{End}}_{kG}(V)=k$.
\begin{enumerate}
\item[i.] Then $\underline{\mathrm{End}}_{kG}(\Omega(V))=k$, and $R(G,V)$ and $R(G,\Omega(V))$ 
are isomorphic.
\item[ii.] There is a non-projective indecomposable $kG$-module $V_0$ $($unique up to
iso\-mor\-phism$)$ such that $\underline{\mathrm{End}}_{kG}(V_0)=k$, $V$ is isomorphic to 
$V_0\oplus P$ for some projective $kG$-module $P$, and $R(G,V)$ and $R(G,V_0)$ are 
isomorphic.
\end{enumerate}
\end{lemma}

\begin{lemma}
\label{lem:Wlift}
{\rm \cite[Lemma 2.3.2]{3sim}}
Let $V$ be a finitely generated $kG$-module such that there is a non-split short exact sequence of 
$kG$-modules
$$0\to Y_2\to V\to Y_1\to 0$$
with $\mathrm{Ext}^1_{kG}(Y_1,Y_2)=k$.
Suppose that  there exists a $WG$-module $X_i$ which is a lift of $Y_i$ over $W$ for $i=1,2$. 
Suppose further that 
$$\mathrm{dim}_F\;\mathrm{Hom}_{FG}(F\otimes_WX_1,F\otimes_WX_2) =
\mathrm{dim}_k\;\mathrm{Hom}_{kG}(Y_1,Y_2)-1.$$
Then there exists a $WG$-module $X$ which is a lift of $V$ over $W$.
\end{lemma}


\section{The dihedral $2$-groups $D$}
\label{s:dihedral}
\setcounter{equation}{0}

Let $d\ge 3$ and let $D$ be a dihedral group of order $2^d$ given as
$$D=\langle \sigma,\tau\;|\; \sigma^2=1=\tau^2, (\sigma\tau)^{2^{d-2}}= (\tau\sigma)^{
2^{d-2}}\rangle.$$
Let $k$ be an algebraically closed field of characteristic $p=2$.
The trivial simple $kD$-module $k$ is the unique irreducible $kD$-module up
to isomorphism. The free $kD$-module $kD$ of rank one is indecomposable and its radical
series has length $2^{d-1}+1$. The radical of $kD$ is generated as a $kD$-module
by $(1+\sigma)$ and by $(1+\tau)$, and the socle of $kD$ is one-dimensional over $k$ and generated
by $[(1+\sigma)(1+\tau)]^{2^{d-2}}=[(1+\tau)(1+\sigma)]^{2^{d-2}}$. Hence
\begin{eqnarray}
\label{eq:radsoc}
\mathrm{rad}(kD)&=&kD(1+\sigma) + kD(1+\tau),\\
\mathrm{soc}(kD) &=& kD(1+\sigma) \cap kD(1+\tau) \;= \;k\,[(1+\sigma)(1+\tau)]^{2^{d-2}}.\nonumber
\end{eqnarray}
From this description it follows that $\mathrm{rad}(kD)/\mathrm{soc}(kD)$ is isomorphic to the
direct sum of two indecomposable $kD$-modules, namely
\begin{equation}
\label{eq:heart}
\mathrm{rad}(kD)/\mathrm{soc}(kD)\cong kD(1+\sigma)/\mathrm{soc}(kD) \;\oplus\; kD(1+\tau)/
\mathrm{soc}(kD).
\end{equation}
Moreover, we have the following isomorphisms of $kD$-modules:
\begin{eqnarray}
\label{eq:important1}
kD(1+\sigma) &\cong & kD\otimes_{k\langle\sigma\rangle} k \;=\;\mathrm{Ind}_{\langle\sigma\rangle}^D\,
k,\\ 
kD(1+\tau) &\cong & kD\otimes_{k\langle\tau\rangle} k \;=\;\mathrm{Ind}_{\langle\tau\rangle}^D\,k.
\nonumber
\end{eqnarray}
Let $\nu\in\{\sigma,\tau\}$, and define $E_\nu=kD(1+\nu)/\mathrm{soc}(kD)$. We have a
commutative diagram of $kD$-modules of the form
\begin{equation}
\label{eq:important2}
\xymatrix{
0\ar[r]&\Omega(E_\nu)\ar[r]^{\iota_\nu}\ar[d]_{\Omega(f_\nu)}&kD\ar[r]^{\pi_\nu}\ar[d]^{g_\nu}&
E_\nu\ar[r]\ar[d]^{f_\nu}&0 \\
0\ar[r]&\mathrm{soc}(kD)\ar[r]^(.55)\iota&kD\ar[r]^(.35)\pi& kD/\mathrm{soc}(kD)\ar[r]&0}
\end{equation}
where $\pi_\nu(1)=(1+\nu)+\mathrm{soc}(kD)$, $\pi(1)=1+\mathrm{soc}(kD)$, 
$\iota_\nu$ and $\iota$ are inclusions, 
$g_\nu(1)=(1+\nu)$ and $f_\nu$ is induced by the inclusion map 
$kD(1+\nu)\hookrightarrow kD$.
Since $f_\nu$ is injective, it follows  that 
\begin{equation}
\label{eq:important3}
\mathrm{Ker}(\Omega(f_\nu))\cong \mathrm{Ker}(g_\nu)=kD(1+\nu).
\end{equation}

We now turn to representations of $D$ in characteristic $0$.
Let $W$ be the ring of infinite Witt vectors over $k$, and let $F$ be the fraction field of $W$.
Let $\zeta$ be a fixed primitive $2^{d-1}$-th root of unity in 
an algebraic closure of $F$. 
Then $D$ 
has $4+(2^{d-2}-1)$ 
ordinary irreducible characters $\chi_1,\chi_2, \chi_3,\chi_4,\chi_{5,i}$, $1\le i\le 2^{d-2}-1$,
whose representations $\psi_1,\psi_2, \psi_3,\psi_4,\psi_{5,i}$, $1\le i\le 2^{d-2}-1$,
are described in Table \ref{tab:chartable}.
\begin{table}[ht] 
\caption{\label{tab:chartable}The ordinary irreducible representations of $D$.}
$$\begin{array}{|c||c|c|}
\hline
 & \;\;\,\sigma&\;\;\,\tau\\\hline
\hline
\psi_1=\chi_1&\;\;\,1&\;\;\,1\\\hline
\psi_2=\chi_2 &-1&-1\\\hline
\psi_3=\chi_3 &\;\;\,1&-1\\\hline
\psi_4=\chi_4 &-1&\;\;\,1\\\hline
\begin{array}{c}\psi_{5,i}\\(1\le i\le 2^{d-2}-1)\end{array}
 &  \left(\begin{array}{cc} 0 & 1 \\ 1&0\end{array}
\right) &\left(\begin{array}{cc}0&\zeta^{-i} \\  \zeta^{i}&0\end{array}
\right)\\\hline
\end{array}$$
\end{table}
In fact, the splitting field of $D$ is $F(\zeta+\zeta^{-1})$, and the action of the Galois group
$\mathrm{Gal}(F(\zeta+\zeta^{-1})/F)$ on the ordinary irreducible characters divides the 
characters $\chi_{5,i}$, $i=1,\ldots,2^{d-2}-1$,
into $d-2$ Galois orbits $\mathcal{O}_0,\ldots, \mathcal{O}_{d-3}$ with
$\mathcal{O}_{\ell}=\{ \chi_{5,2^{d-3-\ell}(2u-1)} \;|\; 1\le u\le 2^{\ell}\}$ for $0\le\ell\le d-3$. 
Since the Schur index over $F$ of each of the characters in these orbits is $1$,
we obtain $d-2$ non-isomorphic simple $FD$-modules $V_0,\ldots,V_{d-3}$
whose characters $\rho_0,\ldots, \rho_{d-3}$ satisfy
\begin{equation}
\label{eq:goodchar1}
\rho_\ell =\sum_{u=1}^{2^{\ell}} \chi_{5,2^{d-3-\ell}(2u-1)} 
\qquad\mbox{for $0\le \ell \le d-3$.}
\end{equation}
Moreover,
\begin{equation}
\label{eq:goodendos1}
\mathrm{End}_{FD}(V_\ell)\cong 
F\left(\zeta^{2^{d-3-\ell}}+\zeta^{-2^{d-3-\ell}}\right)
\qquad \mbox{for $0\le \ell \le d-3$.}
\end{equation}

\begin{lemma}
\label{lem:maxliftsW}
Let 
\begin{itemize}
\item[i.] $(a,b)\in\{(1,3), (2,4)\}$, or
\item[ii.] $(a,b)\in\{(1,4),(2,3)\}$.
\end{itemize}
Let $N_{a,b}$ be an $FD$-module with character 
$\chi_a+\chi_b+\sum_{\ell=0}^{d-3}\rho_\ell$. Then there is a full $WD$-lattice $L_{a,b}$ in $N_{a,b}$
such that $L_{a,b}/2L_{a,b}$ is isomorphic to the 
$kD$-module $\mathrm{Ind}_{\langle\sigma\rangle}^D\,k$ in case $(i)$ and to the $kD$-module 
$\mathrm{Ind}_{\langle\tau\rangle}^D\,k$ in case $(ii)$.
\end{lemma}

\begin{proof}
Let $X_\sigma$ be the set of left cosets of $\langle \sigma \rangle$ in $D$,
and let $W X_\sigma$ be the permutation module of $D$ over $W$ corresponding to $X_\sigma$.
This means that $W X_\sigma$ is a free $W$-module with basis $\{m_x\;|\;x\in X_\sigma\}$ and $g\in D$
acts as $g\cdot m_x=m_{gx}$ for all $x\in X_\sigma$. Using the formula for the permutation character
associated to the $FD$-module $F X_\sigma$, we see that this character is equal to
$\chi_1+\chi_3+\sum_{\ell=0}^{d-3}\rho_\ell$. 

There is a surjective $WD$-module homomorphism
$h_\sigma: WD \to W X_\sigma$ which is defined by $h_\sigma(1)=m_{\langle \sigma\rangle}$.
Then $\mathrm{Ker}(h_\sigma)=WD(1-\sigma)$ and we have a short exact sequence of 
$WD$-modules which are free over $W$
\begin{equation}
\label{eq:easy1}
0\to WD(1-\sigma) \to WD \xrightarrow{h_\sigma} W X_\sigma \to 0.
\end{equation}
Because the character of $FD$ is $\chi_1+\chi_2+\chi_3+\chi_4+2\,\sum_{\ell=0}^{d-3}\rho_\ell$,
the character of $FD(1-\sigma)$ must be 
$\chi_2+\chi_4+\sum_{\ell=0}^{d-3}\rho_\ell$.
Tensoring $(\ref{eq:easy1})$ with $k$ over $W$, we obtain a short exact sequence of $kD$-modules
\begin{equation}
\label{eq:easy2}
0\to kD(1-\sigma) \to kD \to k X_\sigma \to 0.
\end{equation} 
Since $k X_\sigma\cong \mathrm{Ind}_{\langle\sigma\rangle}^D\,k$ and
the latter is isomorphic to $kD(1+\sigma)$ by $(\ref{eq:important1})$, 
Lemma \ref{lem:maxliftsW} follows in case (i). Case (ii) is proved using the set $X_\tau$
of left cosets of $\langle\tau\rangle$ in $D$ instead.
\end{proof}


\section{Endo-trivial modules for $D$ in characteristic $2$}
\label{s:endotrivial}
\setcounter{equation}{0}

As before, let $k$ be an algebraically closed field of characteristic $2$. Since $D$ is a
$2$-group, it follows from \cite{carl1} that the $kD$-modules $V$ with stable endomorphism ring 
$\underline{\mathrm{End}}_{kD}(V)\cong k$ are precisely the endo-trivial $kD$-modules, i.e. the 
$kD$-modules $V$ whose endomorphism ring over $k$, $\mathrm{End}_k(V)$, is as
$kD$-module stably isomorphic to the trivial $kD$-module $k$. The latter modules have been
completely classified in \cite{CT} (see also \cite{carl2}). 
We will use this description to determine 
the location of the indecomposable endo-trivial $kD$-modules in the stable Auslander-Reiten
quiver of $kD$.

\begin{rem}
\label{rem:allIneed}
Let $z=(\sigma\tau)^{2^{d-2}}$ be the involution in the center of $D$.
The poset of all elementary abelian subgroups of $D$ of rank at least 2 consists of two
conjugacy classes of maximal elementary abelian subgroups of rank exactly 2.
These conjugacy classes are
represented by $K_1=\langle \sigma, z\rangle$ and $K_2=\langle \tau,z\rangle$. 
Let $T(D)$ denote the group of equivalence classes of 
endo-trivial $kD$-modules as in \cite{CT}. Consider the map
\begin{equation}
\label{eq:Xi}
\Xi:T(D)\to \mathbb{Z}\times \mathbb{Z}
\end{equation}
defined by $\Xi([M])=(a_1,a_2)$ when $\mathrm{Res}^D_{K_i}M\cong
\Omega^{a_i}_{K_i}(k)\oplus F_{M,i}$ for some free $kK_i$-module $F_{M,i}$ for $i=1,2$. 
In particular, $\Xi([k])=(0,0)$ and 
$\Xi([\Omega^m(M)])=\Xi([M])+(m,m)$ for all endo-trivial $kD$-modules $M$ and 
all integers $m$. 
By \cite[Thm. 5.4]{CT}, $\Xi$ is injective and the image of $\Xi$ is generated by $(1,1)$ and 
$(1,-1)$ (and also by $(1,1)$ and $(-1,1)$). 

As in \S\ref{s:dihedral}, let $E_\sigma=kD(1+\sigma)/\mathrm{soc}(kD)$ and let
$E_\tau=kD(1+\tau)/\mathrm{soc}(kD)$.
By $(\ref{eq:heart})$,
$\mathrm{rad}(kD)/\mathrm{soc}(kD)\cong E_\sigma\oplus E_\tau$.
The almost split sequence ending in $\Omega^{-1}(k)=kD/\mathrm{soc}(kD)$ has thus the form
\begin{equation}
\label{eq:oh1}
0\to \Omega(k)\to kD\oplus E_\sigma\oplus E_\tau \xrightarrow{\mu_{-1}} \Omega^{-1}(k)\to 0
\end{equation}
where $\mu_{-1}|_{E_\nu}$ is the rightmost vertical homomorphism $f_\nu$ 
in the diagram $(\ref{eq:important2})$ and $\mu_{-1}|_{kD}$ is the
natural projection. 
It follows for example from \cite[Lemma 5.4]{AC} that $E_\sigma$ and $E_\tau$ are endo-trivial.
Moreover, $\Xi([E_\sigma])=(1,-1)$ and $\Xi([E_\tau])=(-1,1)$.
In particular,  $T(D)$ is generated by $[\Omega(k)]$ and $[E_\sigma]$ (and
also by $[\Omega(k)]$ and $[E_\tau]$).

Let $\nu\in\{\sigma,\tau\}$, and define $A_{\nu,0}=k$ and $A_{\nu,1}=\Omega(E_\nu)$.
For $n\ge 2$, define $A_{\nu,n}$ to be the unique indecomposable $kD$-module, up to 
isomorphism, in the equivalence class of the endo-trivial $kD$-module 
$A_{\nu,1}\otimes_k A_{\nu,n-1}$.
Then the trivial simple $kD$-module $k=A_{\sigma,0}=A_{\tau,0}$ together with the $kD$-modules 
$A_{\sigma,n}$, $A_{\tau,n}$ for $n\ge 1$ give a complete set of representatives of the $\Omega$-orbits 
of the indecomposable endo-trivial $kD$-modules. We have
$\Xi([A_{\sigma,n}])=(2n,0)$ and $\Xi([A_{\tau,n}])=(0,2n)$ for all $n\ge 0$.
\end{rem}

\begin{lemma}
\label{lem:endotrivial}
The finitely generated indecomposable endo-trivial  $kD$-modules are exactly the modules in the 
two components of the stable Auslander-Reiten quiver of $kD$ containing the trivial simple 
$kD$-module $k$ and $\Omega(k)$.

More precisely, 
let $A_{\sigma,n}$ and $A_{\tau,n}$ be as in Remark $\ref{rem:allIneed}$ for $n\ge 0$.
Then the almost split sequence ending in $k$ has the form
\begin{equation}
\label{eq:assk}
0\to \Omega^2(k)\to A_{\sigma,1}\oplus A_{\tau,1}\xrightarrow{\mu_1} k\to 0.
\end{equation}
Let $\nu\in\{\sigma,\tau\}$ and let $n\ge 1$. Then the
almost split sequence ending in $A_{\nu,n}$ has the form
\begin{equation}
\label{eq:assgeneral}
0\to \Omega^2(A_{\nu,n}) \to A_{\nu,n+1}\oplus \Omega^2(A_{\nu,n-1}) \xrightarrow{\mu_{\nu,n+1}} 
A_{\nu,n}\to 0.
\end{equation}
\end{lemma}

\begin{proof}
Since $\Omega$ defines an equivalence of the stable module category of finitely generated
$kD$-modules with itself, we can apply $\Omega$ to the almost split sequence
$(\ref{eq:oh1})$ to obtain the
almost split sequence ending in $k$ up to free direct summands of the middle term.
Since the sequence $(\ref{eq:oh1})$ is the only almost split sequence having $kD$ as 
a summand of the middle term,
the almost split sequence ending in $k$  is as in $(\ref{eq:assk})$.

Given an indecomposable $kD$-module $M$ of odd $k$-dimension, it follows from
\cite[Thm. 3.6 and Cor. 4.7]{AC} that 
\begin{equation}
\label{eq:oh3}
0\to \Omega^2(k)\otimes_k M\to  (A_{\sigma,1}\otimes_k M)\oplus (A_{\tau,1}\otimes_k M) \to M\to 0
\end{equation}
is the almost split sequence ending in $M$ modulo projective direct summands.
Since all endo-trivial $kD$-modules have odd $k$-dimension,
we can apply the sequence $(\ref{eq:oh3})$ to $M=A_{\nu,n}$ for $\nu\in\{\sigma,\tau\}$ and
all $n\ge 1$. This means that modulo free direct summands the almost split sequence ending in 
$A_{\nu,n}$ has the form
\begin{equation}
\label{eq:oh4}
0\to \Omega^2(k)\otimes_k A_{\nu,n}\to  (A_{\sigma,1}\otimes_k A_{\nu,n})\oplus 
(A_{\tau,1}\otimes_k A_{\nu,n}) \to A_{\nu,n}\to 0.
\end{equation}
Note that 
$\Xi([A_{\nu',1}\otimes_kA_{\nu,n}])=(2,2)+\Xi([A_{\nu,n-1}])=\Xi([\Omega^2(A_{\nu,n-1})])$
if $\{\nu,\nu'\}=\{\sigma,\tau\}$.
Since the sequence $(\ref{eq:oh1})$ is the only almost split sequence having $kD$ as a summand of the 
middle term, it follows that the almost split sequence ending in $A_{\nu,n}$ is as in 
$(\ref{eq:assgeneral})$. This completes the proof of Lemma \ref{lem:endotrivial}.
\end{proof}

\begin{lemma}
\label{lem:thiswilldo}
Let $\nu\in\{\sigma,\tau\}$, let $n\ge 0$ and let $A_{\nu,n}$ be as in Remark $\ref{rem:allIneed}$.
Then $\mathrm{dim}_k\,A_{\nu,n}=n2^{d-1}+1$ and $\mathrm{Res}_C^D\,A_{\nu,n}\cong
k\oplus (kC)^{n2^{d-2}}$ for $C\in\{\langle\sigma\rangle,\langle\tau\rangle\}$.
Moreover, there is a short exact sequence of $kD$-modules
\begin{equation}
\label{eq:ohwell}
0\to \mathrm{Ind}_{\langle \nu\rangle}^D\, k\to A_{\nu,n+1}\to A_{\nu,n}\to 0.
\end{equation}
\end{lemma}

\begin{proof}
When we restrict the almost split sequences $(\ref{eq:assk})$ and
$(\ref{eq:assgeneral})$ to the elementary abelian subgroups $K_1$ and $K_2$ from
Remark $\ref{rem:allIneed}$, it follows that the resulting short exact sequences of
$kK_i$-modules split for $i=1,2$. 
Define $\phi_{\nu,1}:A_{\nu,1}\to k$ to be the restriction of the homomorphism $\mu_1$ in 
$(\ref{eq:assk})$ to the component $A_{\nu,1}$, and for $n\ge 1$ define $\phi_{\nu,n+1}:A_{\nu,n+1}\to 
A_{\nu,n}$ to be the restriction of the homomorphism $\mu_{\nu,n+1}$ in $(\ref{eq:assgeneral})$ to the 
component $A_{\nu,n+1}$.
For $n\ge 1$, let $\Phi_{\nu,n}:A_{\nu,n}\to k$ be the composition $\Phi_{\nu,n}=
\phi_{\nu,1}\circ\phi_{\nu,2}\circ\ldots\circ\phi_{\nu,n}$. Then it follows that the homomorphism
$\Phi_n:A_{\sigma,n}\oplus A_{\tau,n}\to k$, 
which restricted to $A_{\nu,n}$ is given by $\Phi_{\nu,n}$,
splits when viewed as a homomorphism of $kK_i$-modules for $i=1,2$. 
Since $\Xi([A_{\sigma,n}])=(2n,0)$ and $\Xi([A_{\tau,n}])=(0,2n)$, this implies
that for all $n\ge 1$ we have a short exact sequence of $kD$-modules of the form
$$0\to \Omega^{2n}(k)\to A_{\sigma,n}\oplus A_{\tau,n}\xrightarrow{\Phi_n} k \to 0.$$
Inductively, we see that $\Omega^{2n}(k)$ has $k$-dimension $n2^d+1$. Because 
the $k$-dimensions of $A_{\sigma,n}$ and $A_{\tau,n}$ coincide,
this implies that
$\mathrm{dim}_k\,A_{\nu,n}=\frac{1}{2}\,(n2^d+2)=n2^{d-1}+1$.

Let $C\in\{\langle\sigma\rangle,\langle\tau\rangle\}$.
Since $E_\nu=kD(1+\nu)/\mathrm{soc}(kD)$, it follows that $\mathrm{Res}_C^D\, E_\nu$
is stably isomorphic to $k$. Hence we obtain for  $A_{\nu,1}=\Omega(E_\nu)$ that  
$\mathrm{Res}_C^D\, A_{\nu,1}$ also is stably isomorphic to $k$. 
Therefore, it follows by induction from the almost split sequence $(\ref{eq:assgeneral})$ that
$\mathrm{Res}_C^D\, A_{\nu,n+1}$ is stably isomorphic to $k$ for all $n\ge 1$. Comparing 
$k$-dimensions it follows that
$\mathrm{Res}_C^D\, A_{\nu,n}\cong k\oplus
(kC)^{n2^{d-2}}$ for all $n\ge 0$.

To construct a short exact sequence of the form $(\ref{eq:ohwell})$, recall that
the almost split sequence $(\ref{eq:assk})$ is obtained by applying $\Omega$
to the almost split sequence $(\ref{eq:oh1})$. Since the restriction of the homomorphism
$\mu_{-1}$ in $(\ref{eq:oh1})$ to the component $E_\nu$ is the same as the 
homomorphism $f_\nu$ in the diagram $(\ref{eq:important2})$,
the restriction of the homomorphism $\mu_1$ in 
$(\ref{eq:assk})$ to the component $A_{\nu,1}$ is the same as
$\Omega(f_\nu)$. By  $(\ref{eq:important3})$ we have $\mathrm{Ker}(\Omega(f_\nu))\cong
kD(1+\nu)$, which implies by $(\ref{eq:important1})$ that there is 
a short exact sequence of $kD$-modules of the form
\begin{equation}
\label{eq:labelit1}
0\to \mathrm{Ind}_{\langle \nu\rangle}^D\, k\to A_{\nu,1}\to k\to 0.
\end{equation}
Now let $n\ge 1$. Tensoring the sequence $(\ref{eq:labelit1})$ with $A_{\nu,n}$, we obtain
a short exact sequence of $kD$-modules of the form
\begin{equation}
\label{eq:labelit2}
0\to (\mathrm{Ind}_{\langle \nu\rangle}^D\, k)\otimes_k A_{\nu,n}\to A_{\nu,1}\otimes_k A_{\nu,n}\to 
A_{\nu,n}\to 0.
\end{equation}
Since $(\mathrm{Ind}_{\langle \nu\rangle}^D\, k)\otimes_k A_{\nu,n}\cong
\mathrm{Ind}_{\langle \nu\rangle}^D\,(\mathrm{Res}_{\langle \nu\rangle}^D\, A_{\nu,n})$
and $\mathrm{Res}_{\langle \nu\rangle}^D\, A_{\nu,n}\cong k\oplus (k\langle\nu\rangle)^{n2^{d-2}}$, 
it follows  that
$$(\mathrm{Ind}_{\langle \nu\rangle}^D\, k)\otimes_k A_{\nu,n}
\cong \mathrm{Ind}_{\langle \nu\rangle}^D\, k\;\oplus (kD)^{n2^{d-2}}.$$
By definition, $A_{\nu,n+1}$ is the unique indecomposable $kD$-module, up to 
isomorphism, in the equivalence class of the endo-trivial $kD$-module $A_{\nu,1}\otimes_k A_{\nu,n}$.
Comparing $k$-dimensions it follows that
$\mathrm{dim}_k\,A_{\nu,1}\otimes_k A_{\nu,n}=
n2^{d-2}\cdot 2^d+\mathrm{dim}_k\,A_{\nu,n+1}$.
Hence
$$A_{\nu,1}\otimes_k A_{\nu,n}\cong A_{\nu,n+1}\oplus (kD)^{n2^{d-2}}.$$
By splitting off the free $kD$-module $(kD)^{n2^{d-2}}$ from the first and second term of
the short exact sequence $(\ref{eq:labelit2})$
we obtain a short exact sequence of $kD$-modules of the form $(\ref{eq:ohwell})$.
\end{proof}


\section{Universal deformation rings for $D$ and nilpotent blocks}
\label{s:udr}
\setcounter{equation}{0}

In this section, we prove Theorem \ref{thm:supermain} and Corollary \ref{cor:nilpotent}. 
We use the notations of the previous
sections. In particular, $k$ is an algebraically closed field of characteristic $2$, $W$ is the
ring of infinite Witt vectors over $k$ and $F$ is the fraction field of $W$.

For $\nu\in\{\sigma,\tau\}$ and $n\ge 0$, let $A_{\nu,n}$ be as in Remark \ref{rem:allIneed}.
We first  analyze all extensions of $A_{\nu,n}$ by $A_{\nu,n}$ by using the extensions of $k$
by $k$ which are described in the following remark.

\begin{rem}
\label{rem:dothisfirst}
Suppose $N$ is a $kD$-module which lies in a non-split short exact sequence of $kD$-modules
\begin{equation}
\label{eq:sesk}
0\to k\to N\to k\to 0.
\end{equation}
Then $N$ is isomorphic to $N_\lambda$ for some $\lambda\in k^*\cup \{\sigma,\tau\}$ where
a representation $\varphi_\lambda$ of $N_\lambda$ is given by the following $2\times 2$ matrices
over $k$.
\begin{itemize}
\item[a.] If $\lambda \in k^*$ then $\varphi_\lambda(\sigma)=\left(\begin{array}{cc}1&1\\0&1
\end{array}\right)$,
$\varphi_\lambda(\tau)=\left(\begin{array}{cc}1&\lambda\\0&1\end{array}\right)$.
\vspace{1ex}
\item[b.] If $\{\lambda,\lambda'\}=\{\sigma,\tau\}$ then $\varphi_\lambda(\lambda)=\left(\begin{array}{cc}1&1\\0&1\end{array}\right)$,
$\varphi_\lambda(\lambda')=\left(\begin{array}{cc}1&0\\0&1\end{array}\right)$.
\end{itemize}
We obtain the following restrictions of $N_\lambda$ to the subgroups $\langle\sigma\rangle$
and $\langle\tau\rangle$ of $D$.
\begin{itemize}
\item In case (a),
$\mathrm{Res}^D_C\, N_\lambda\cong kC$ for $C\in\{\langle\sigma\rangle,\langle\tau\rangle\}$.
\vspace{1ex}
\item In case (b),
$\mathrm{Res}^D_{\langle\lambda\rangle}\, N_\lambda\cong k\langle\lambda\rangle$
and $\mathrm{Res}^D_{\langle\lambda'\rangle}\, N_\lambda\cong k^2$.
\end{itemize}
\end{rem}

\begin{lemma}
\label{lem:extensions}
Let  $\nu\in\{\sigma,\tau\}$, let $n\ge 0$, and let $A_{\nu,n}$ be as in Remark $\ref{rem:allIneed}$.
Then $\mathrm{Ext}^1_{kD}(A_{\nu,n},A_{\nu,n})\cong k^2$.
Suppose $Z_{\nu,n}$ is a $kD$-module which lies in a non-split
short exact sequence of $kD$-modules
\begin{equation}
\label{eq:ext1}
0\to A_{\nu,n}\to Z_{\nu,n}\to A_{\nu,n}\to 0.
\end{equation}
Then $Z_{\nu,n}$ is isomorphic to $B_{\nu,n,\lambda}=A_{\nu,n}\otimes_k N_\lambda$ for some 
$\lambda\in k^*\cup\{\sigma,\tau\}$ where $N_\lambda$ is as in Remark $\ref{rem:dothisfirst}$.
Let $\{\nu,\nu'\}=\{\sigma,\tau\}$.
\begin{itemize}
\item[i.] If $\lambda\in k^*\cup\{\nu\}$ then $B_{\nu,n,\lambda}\cong N_\lambda\oplus (kD)^n$.
\item[ii.] If $\lambda=\nu'$ then $B_{\nu,0,\nu'}\cong N_{\nu'}$
and for $n\ge 1$, $B_{\nu,n,\nu'}$ is a non-split extension of $B_{\nu,n-1,\nu'}$ by $kD(1+\nu)\oplus
kD(1+\nu)$. Moreover, 
$\mathrm{Res}^D_{\langle\nu'\rangle}\, B_{\nu,n,\nu'}\cong (k\langle\nu'\rangle)^{n2^{d-1}+1}$ and
$\mathrm{Res}^D_{\langle\nu\rangle}\, B_{\nu,n,\nu'}\cong k^2\oplus (k\langle\nu\rangle)^{n2^{d-1}}$.
\end{itemize}
\end{lemma}

\begin{proof}
Since $A_{\nu,n}$ is endo-trivial, it follows from \cite[Thm. 2.6]{AC} that the natural homomorphism
\begin{equation}
\label{eq:does}
\mathrm{Ext}^1_{kD}(k,k)\to \mathrm{Ext}^1_{kD}(A_{\nu,n},A_{\nu,n})
\end{equation}
resulting from tensoring  short exact sequences of the form $(\ref{eq:sesk})$ with
$A_{\nu,n}$ is a monomorphism. Moreover,
$$\mathrm{Ext}_{kD}^1(A_{\nu,n},A_{\nu,n})\cong 
\HH^1(D,A_{\nu,n}^*\otimes_kA_{\nu,n})
\cong \HH^1(D,k)\cong \mathrm{Ext}^1_{kD}(k,k)\cong k^2$$
where the second isomorphism follows since $A_{\nu,n}$ is endo-trivial. Hence the homomorphism
in $(\ref{eq:does})$ is an isomorphism, which means that we only need 
to prove the descriptions of $B_{\nu,n,\lambda}=A_{\nu,n}\otimes_kN_\lambda$ given in (i) and (ii) of 
the statement of Lemma \ref{lem:extensions}.
For  $n=0$ this follows from Remark \ref{rem:dothisfirst} since $A_{\nu,0}=k$ and so
$B_{\nu,0,\lambda}=N_\lambda$.

By Lemma \ref{lem:thiswilldo}, there is a short exact sequence of $kD$-modules of the form
$(\ref{eq:ohwell})$. Tensoring this sequence with $N_\lambda$ over $k$ gives a short exact
sequence of $kD$-modules of the form
\begin{equation}
\label{eq:lala1}
0\to (\mathrm{Ind}_{\langle \nu\rangle}^D\, k)\otimes_kN_\lambda \to B_{\nu,n+1,\lambda}
\to B_{\nu,n,\lambda}\to 0
\end{equation}
where $(\mathrm{Ind}_{\langle \nu\rangle}^D\, k)\otimes_kN_\lambda\cong
\mathrm{Ind}_{\langle \nu\rangle}^D\,(\mathrm{Res}_{\langle\nu\rangle}^D\,N_\lambda)$.

Let first $\lambda\in k^*\cup\{\nu\}$. Then $\mathrm{Res}_{\langle\nu\rangle}^D\,N_\lambda
\cong k\langle \nu\rangle$ 
and hence $\mathrm{Ind}_{\langle \nu\rangle}^D\,(\mathrm{Res}_{\langle\nu\rangle}^D\,N_\lambda)
\cong kD$. It follows by induction that  $B_{\nu,n+1,\lambda}
\cong N_\lambda\oplus (kD)^{n+1}$, which proves (i).

Now let $\lambda=\nu'$. Then $\mathrm{Res}_{\langle\nu\rangle}^D\,N_{\nu'}\cong k^2$ and hence 
$$\mathrm{Ind}_{\langle \nu\rangle}^D\,(\mathrm{Res}_{\langle\nu\rangle}^D
\,N_\nu')\cong \mathrm{Ind}_{\langle \nu\rangle}^D\,k^2\cong kD(1+\nu)\oplus kD(1+\nu).$$
Thus $B_{\nu,n+1,\nu'}$ is an extension of $B_{\nu,n,\nu'}$ by $kD(1+\nu)\oplus
kD(1+\nu)$.
Since $\mathrm{Res}_{\langle \nu'\rangle}^D\,kD(1+\nu)\cong (k\langle \nu'\rangle)^{2^{d-2}}$,
it follows by induction that 
$\mathrm{Res}^D_{\langle \nu'\rangle}\, B_{\nu,n+1,\nu'}\cong (k\langle \nu'\rangle)^{(n+1)2^{d-1}+1}$.
On the other hand, $\mathrm{Res}_{\langle \nu\rangle}^D\,kD(1+\nu)\cong k^2\oplus (k\langle \nu\rangle)^{2^{d-2}-1}$.
Hence the restriction of the left term in the short exact sequence $(\ref{eq:lala1})$ to $\langle \nu\rangle$
is isomorphic to $k^4\oplus (k\langle \nu\rangle)^{2^{d-1}-2}$. Thus by induction, 
$\mathrm{Res}^D_{\langle \nu\rangle}\, B_{\nu,n+1,\nu'}$ lies in a short exact sequence of $k\langle \nu\rangle$-modules
of the form
$$0\to k^4\oplus (k\langle \nu\rangle)^{2^{d-1}-2}\to \mathrm{Res}^D_{\langle \nu\rangle}\, B_{\nu,n+1,\nu'}
\to k^2\oplus (k\langle \nu\rangle)^{n2^{d-1}}\to 0.$$
Since $\mathrm{Res}^D_{\langle \nu\rangle}\, B_{\nu,n+1,\nu'}$ is an extension of 
$\mathrm{Res}^D_{\langle \nu\rangle}\, A_{\nu,n+1}$ by itself and since $\mathrm{Res}^D_{\langle \nu\rangle}\, A_{\nu,n+1}
\cong k\oplus (k\langle \nu\rangle)^{(n+1)2^{d-2}}$
by Lemma \ref{lem:thiswilldo}, it follows that $\mathrm{Res}^D_{\langle \nu\rangle}\, B_{\nu,n+1,\nu'}
\cong k^2\oplus (k\langle \nu\rangle)^{(n+1)2^{d-1}}$.
In particular, the sequence $(\ref{eq:lala1})$ does not split when $\lambda=\nu'$, since it does not split when restricted to
$\langle \nu\rangle$. This proves (ii).
\end{proof}

The next result uses the restrictions of $A_{\nu,n}$ and $B_{\nu,n,\lambda}$ to the cyclic subgroups 
$\langle\sigma\rangle$ and $\langle\tau\rangle$ of $D$.

\begin{lemma}
\label{lem:almostudr}
Let  $\nu\in\{\sigma,\tau\}$, let $n\ge 0$, and let $A_{\nu,n}$ be as in Remark $\ref{rem:allIneed}$.
Then there is a surjective $W$-algebra homomorphism 
$\alpha: W[\mathbb{Z}/2\times\mathbb{Z}/2]\to R(D,A_{\nu,n})$ in $\mathcal{C}$.
\end{lemma}

\begin{proof}
Let $R=R(D,A_{\nu,n})$ and let $U_{\nu,n}$ be a universal lift of 
$A_{\nu,n}$ over $R$. Let $C_\sigma=\langle\sigma\rangle$ and 
$C_\tau=\langle\tau\rangle$, and let $C\in\{C_\sigma,C_\tau\}$. By Lemma \ref{lem:thiswilldo},
$\mathrm{Res}^D_{C}A_{\nu,n}\cong k\oplus (kC)^{n2^{d-2}}$. 
In particular, $\mathrm{Res}^D_{C}A_{\nu,n}$  is a $kC$-module with stable endomorphism ring $k$,
and hence it has a universal deformation ring. Moreover, 
$R(C,\mathrm{Res}^D_{C}A_{\nu,n})\cong W[\mathbb{Z}/2]$.
Let $U_{\nu,n,C}$ be a universal
lift of $\mathrm{Res}^D_{C}A_{\nu,n}$ over $W[\mathbb{Z}/2]$. Then 
there exists a unique $W$-algebra homomorphism $\alpha_C:W[\mathbb{Z}/2]\to R$ in $\mathcal{C}$ 
such that $\mathrm{Res}^D_{C}U_{\nu,n}\cong R\otimes_{W[\mathbb{Z}/2],\alpha_C} U_{\nu,n,C}$. 
Since the completed
tensor product over $W$ is the coproduct in the category $\mathcal{C}$, we get a $W$-algebra
homomorphism
$$\alpha=\alpha_{C_\sigma}\otimes \alpha_{C_\tau}: W[\mathbb{Z}/2]\otimes_W W[\mathbb{Z}/2]\to R$$
in $\mathcal{C}$
with $\alpha(x_1\otimes x_2)=\alpha_{C_\sigma}(x_1)\,\alpha_{C_\tau}(x_2)$. 

By Lemma \ref{lem:extensions}, every non-trivial lift of $A_{\nu,n}$ over the dual numbers
$k[\epsilon]/(\epsilon^2)$ is as $kD$-module isomorphic to $B_{\nu,n,\lambda}$ for some
$\lambda\in k^*\cup\{\sigma,\tau\}$. The description of these modules shows that their restrictions 
to $C_\sigma$ and $C_\tau$ are as follows:
\begin{enumerate}
\item[i.] If $\lambda\in k^*$ then 
$\mathrm{Res}^D_C \,B_{\nu,n,\lambda}=(kC)^{n2^{d-1}+1}$ for $C\in\{C_\sigma,C_\tau\}$.
\item[ii.] If $\{\lambda,\lambda'\}=\{\sigma,\tau\}$, then
$\mathrm{Res}^D_{C_{\lambda}}\, B_{\nu,n,\lambda}=(kC_{\lambda})^{n2^{d-1}+1}$
and $\mathrm{Res}^D_{C_{\lambda'}}\, B_{\nu,n,\lambda}=k^2\oplus (kC_{\lambda'})^{n2^{d-1}}$.
\end{enumerate}
Note that for $C\in\{C_\sigma,C_\tau\}$, $k^2\oplus (kC)^{n2^{d-1}}$ corresponds to the trivial lift of 
$\mathrm{Res}^D_{C}A_{\nu,n}$ over $k[\epsilon]/(\epsilon^2)$, and $(kC)^{n2^{d-1}+1}$ 
corresponds to a non-trivial lift  of $\mathrm{Res}^D_{C}A_{\nu,n}$ over $k[\epsilon]/(\epsilon^2)$.

Let $\lambda\in k^*\cup\{\sigma,\tau\}$, and let $f_\lambda:R\to k[\epsilon]/(\epsilon^2)$ be a morphism 
corresponding to a non-trivial lift of $A_{\nu,n}$ over $k[\epsilon]/(\epsilon^2)$ with underlying 
$kD$-module structure given by $B_{\nu,n,\lambda}$. Then 
$f_\lambda\circ \alpha=f_\lambda\circ(\alpha_{C_\sigma}\otimes \alpha_{C_\tau})$ 
corresponds to the pair of lifts
of the pair $(\mathrm{Res}^D_{C_\sigma}A_{\nu,n},\mathrm{Res}^D_{C_\tau}A_{\nu,n})$ over 
$k[\epsilon]/(\epsilon^2)$
with underlying $kD$-module structure given by the pair
$(\mathrm{Res}^D_{C_\sigma} B_{\nu,n,\lambda},\mathrm{Res}^D_{C_\tau} B_{\nu,n,\lambda})$.
Hence (i) and (ii) above imply that if $f$ runs through the morphisms 
$R\to k[\epsilon]/(\epsilon^2)$, then $f\circ \alpha$ runs through the morphisms
$W[\mathbb{Z}/2]\otimes_W W[\mathbb{Z}/2]\to k[\epsilon]/(\epsilon^2)$.
This implies $\alpha$ is surjective.
\end{proof}

We next determine how many non-isomorphic lifts $A_{\nu,n}$ has over $W$.

\begin{lemma}
\label{lem:almostdone}
Let  $\nu\in\{\sigma,\tau\}$, let $n\ge 0$, and let $A_{\nu,n}$ be as in Remark $\ref{rem:allIneed}$. 
Then $A_{\nu,n}$ has four pairwise non-isomorphic
lifts over $W$ corresponding to four distinct morphisms $R(D,A_{\nu,n})\to W$ in $\mathcal{C}$.
\end{lemma}

\begin{proof}
We use Lemmas \ref{lem:Wlift} and \ref{lem:maxliftsW} to prove this.
If $n=0$ then $A_{\nu,0}$ is the trivial simple $kD$-module $k$ which has 
four pairwise non-isomorphic lifts
over $W$ whose $F$-characters are given by the four ordinary irreducible characters of degree one 
$\chi_1,\chi_2,\chi_3,\chi_4$ (see Table \ref{tab:chartable}). 
By Lemma \ref{lem:thiswilldo}, there is a short exact sequence of $kD$-modules of the form
$$0\to \mathrm{Ind}_{\langle\nu\rangle}^D\,k\to A_{\nu,n+1}\to A_{\nu,n}\to 0.$$

By  Frobenius reciprocity and the Eckman-Shapiro Lemma,  we have for all $n\ge 0$
\begin{eqnarray*}
\mathrm{Hom}_{kD}(A_{\nu,n},\mathrm{Ind}_{\langle\nu\rangle}^D\,k)&\cong&
\mathrm{Hom}_{k\langle\nu\rangle}(\mathrm{Res}_{\langle\nu\rangle}^D\,A_{\nu,n},k)
\;\cong \;k^{n2^{d-2}+1}\quad\mbox{and}\\
\mathrm{Ext}^1_{kD}(A_{\nu,n},\mathrm{Ind}_{\langle\nu\rangle}^D\,k)&\cong&
\mathrm{Ext}^1_{k\langle\nu\rangle}(\mathrm{Res}_{\langle\nu\rangle}^D\,A_{\nu,n},k)
\;\cong\; k 
\end{eqnarray*}
where the second isomorphisms follow since $\mathrm{Res}_{\langle\nu\rangle}^D\,A_{\nu,n}
\cong k\oplus (k\langle\nu\rangle)^{n2^{d-2}}$ by Lemma \ref{lem:thiswilldo}.

Let $\rho=\sum_{\ell=0}^{d-3} \rho_\ell$ for $\rho_\ell$ as in $(\ref{eq:goodchar1})$.
Then every $FD$-module $T$ with $F$-character $\rho$ satisfies by $(\ref{eq:goodendos1})$
$$\mathrm{dim}_F\,\mathrm{End}_{FD}(T)=\sum_{\ell=0}^{d-3}\mathrm{dim}_F\,\mathrm{End}_{FD}(
V_\ell) =\sum_{\ell=0}^{d-3}2^\ell = 2^{d-2}-1.$$
Let $(c,d)=(3,4)$ if $\nu=\sigma$, and let $(c,d)=(4,3)$ if $\nu=\tau$.
Assume by induction that $A_{\nu,n}$ has four pairwise non-isomorphic lifts
over $W$ whose $F$-characters are given by 
\begin{itemize}
\item[i.] $\chi_1+m(\chi_1+\chi_c)+m(\chi_2+\chi_d)+2m\,\rho$, or \\
$\chi_c+m(\chi_1+\chi_c)+m(\chi_2+\chi_d)+2m\,\rho$, or
\item[ii.] $\chi_2+m(\chi_1+\chi_c)+m(\chi_2+\chi_d)+2m\,\rho$, or \\
$\chi_d+m(\chi_1+\chi_c)+m(\chi_2+\chi_d)+2m\,\rho$
\end{itemize}
if $n=2m$ for some $m\ge 0$, and by
\begin{itemize}
\item[i$'$.] $\chi_1+m(\chi_1+\chi_c)+(m+1)(\chi_2+\chi_d)+(2m+1)\rho$, or \\
$\chi_c+m(\chi_1+\chi_c)+(m+1)(\chi_2+\chi_d)+(2m+1)\rho$, or
\item[ii$'$.] $\chi_2+(m+1)(\chi_1+\chi_c)+m(\chi_2+\chi_d)+(2m+1)\rho$, or \\
$\chi_d+(m+1)(\chi_1+\chi_c)+m(\chi_2+\chi_d)+(2m+1)\rho$
\end{itemize}
if $n=2m+1$ for some $m\ge 0$.
By Lemma \ref{lem:maxliftsW}, $\mathrm{Ind}_{\langle\nu\rangle}^D\,k$ 
has two pairwise non-isomorphic lifts over $W$ with
$F$-characters $\chi_1+\chi_c+\rho$ or $\chi_2+\chi_d+\rho$.
If $\mathcal{A}_{\nu,n}$ is a lift of $A_{\nu,n}$ over $W$ with $F$-character
as in (i) or (ii$'$), let $\mathcal{I}$ be a lift of $\mathrm{Ind}_{\langle\nu\rangle}^D\,k$ 
over $W$ with $F$-character 
$\chi_2+\chi_d+\rho$. If $\mathcal{A}_{\nu,n}$ is a lift of $A_{\nu,n}$ over $W$ with $F$-character
as in (ii) or (i$'$), let $\mathcal{I}$ be a lift of $\mathrm{Ind}_{\langle\nu\rangle}^D\,k$ over 
$W$ with $F$-character $\chi_1+\chi_c+\rho$.
Then we obtain for $n=2m$,
$$\mathrm{dim}_F\,\mathrm{Hom}_{FD}(F\otimes_W\mathcal{A}_{\nu,n},F\otimes_W\mathcal{I})
=2m\,2^{d-2} =  n2^{d-2},$$
and for $n=2m+1$,
$$\mathrm{dim}_F\,\mathrm{Hom}_{FD}(F\otimes_W\mathcal{A}_{\nu,n},F\otimes_W\mathcal{I})
=(2m+1)\,2^{d-2}=n2^{d-2}.$$
Hence by Lemma \ref{lem:Wlift}, $A_{\nu,n+1}$ has four pairwise non-isomorphic lifts over $W$ whose
$F$-characters are as in (i),(ii), respectively (i$'$), (ii$'$), if we replace $n$ by $n+1$.
\end{proof}

\medskip

\noindent
\textit{Proof of Theorem \ref{thm:supermain}.}
Part (i) follows from Lemma \ref{lem:endotrivial}. 
For part (ii),
let $V$ be an arbitrary finitely generated endo-trivial $kD$-module. By Lemma \ref{lem:defhelp}(ii),
it is enough to determine $R(D,V)$ in case $V$ is indecomposable.
This means by Remark \ref{rem:allIneed} that $V$ is in the $\Omega$-orbit
of $k$ or of $A_{\sigma,n}$ or of $A_{\tau,n}$ for  some $n\ge 1$. 

Let $\nu\in\{\sigma,\tau\}$ and let $n\ge 0$.
It follows by Lemmas 
\ref{lem:almostudr} and \ref{lem:almostdone} that there is a surjective morphism
$$\alpha:W[\mathbb{Z}/2\times\mathbb{Z}/2]\to R(D,A_{\nu,n})$$ in $\mathcal{C}$
and that there are four distinct morphisms $R(D,A_{\nu,n})\to W$ in $\mathcal{C}$.
Hence $\mathrm{Spec}(R(D,A_{\nu,n}))$ contains all four points of the generic 
fiber of $\mathrm{Spec}(W[\mathbb{Z}/2\times\mathbb{Z}/2])$.
Since the Zariski closure of these four points is all of $\mathrm{Spec}(W[\mathbb{Z}/2\times\mathbb{Z}/2])$,
this implies that $R(D,A_{\nu,n})$ must be isomorphic to $W[\mathbb{Z}/2\times\mathbb{Z}/2]$.
By Lemma \ref{lem:defhelp}, it follows that $R(D,V)\cong W[\mathbb{Z}/2\times \mathbb{Z}/2]$
for every finitely generated endo-trivial $kD$-module $V$.

Let $U$ be a universal lift of $V$  over $R=R(D,V)$. 
Since $V$ is endo-trivial, the rank of $U$ as a free $R$-module is odd. This implies that
as $RD$-modules
$$U^*\otimes_R U\cong R\oplus L$$
where $D$ acts trivially on $R$ and $L$ is some $RD$-module which is free over $R$. 
Since $U^*\otimes_R U$ is a lift of  $V^*\otimes_k V$ over $R$ and
$V$ is endo-trivial, this implies that $k\otimes_R L$ is isomorphic to a free $kD$-module.
Hence $L$ is a free $RD$-module, which implies that $U$ is endo-trivial.
\hspace*{\fill} $\Box$

\medskip

We now turn to nilpotent blocks and the proof of Corollary \ref{cor:nilpotent}.

\begin{rem}
\label{rem:nilpotent}
Keeping the previous notation,
let $G$ be a finite group and let $B$ be a nilpotent block of $kG$ with defect group $D$.
By \cite{brouepuig}, this means that whenever $(Q,f)$ is a $B$-Brauer pair then 
the quotient $N_G(Q,f)/C_G(Q)$ 
is a $2$-group. 
In other words,
for all subgroups $Q$ of $D$ and for all block idempotents $f$ of $kC_G(Q)$ associated with $B$, the 
quotient of the stabilizer $N_G(Q,f)$ of $f$ in $N_G(Q)$ by the centralizer $C_G(Q)$
is a $2$-group.
In \cite{puig}, Puig rephrased this definition using the theory of local pointed groups. 

The main result of \cite{puig} implies that the nilpotent block $B$ is Morita equivalent to $kD$. In
\cite[Thm. 8.2]{puig2}, Puig showed that the converse is also true in a very strong way.
Namely, if $B'$ is another block such that there is a stable equivalence of
Morita type between $B$ and $B'$, then $B'$ is also nilpotent. Hence Corollary
\ref{cor:nilpotent} can be applied in particular if there is only known to be a stable
equivalence of Morita type between $B$ and $kD$.
\end{rem}

\noindent
\textit{Proof of Corollary \ref{cor:nilpotent}.}
Let $\hat{B}$ be the block of $WG$ corresponding to $B$. Then $\hat{B}$ is also nilpotent, and
by \cite[\S1.4]{puig}, $\hat{B}$ is Morita equivalent to $WD$. 
Suppose $V$ is a finitely generated $B$-module, and $V'$ is the $kD$-module corresponding to
$V$ under this Morita equivalence.
Then $V$ has stable endomorphism ring $k$ if and only if $V'$ has
stable endomorphism ring $k$. 
Moreover, it follows for example from \cite[Prop. 2.5]{bl} that $R(G,V)\cong R(D,V')$.
By Theorem \ref{thm:supermain},
this implies that $R(G,V)\cong W[\mathbb{Z}/2\times\mathbb{Z}/2]$.
\hspace*{\fill} $\Box$


\end{document}